\documentclass[11pt]{amsart}

\author[C.~Sanna]{Carlo Sanna$^\dagger$}
\thanks{$\dagger\,$C.~Sanna is a member of GNSAGA of INdAM and of CrypTO, the group of Cryptography and Number~Theory of Politecnico di Torino}
\address{\parbox{\linewidth}{
Department of Mathematical Sciences, Politecnico di Torino\\
Corso Duca degli Abruzzi 24, 10129 Torino, Italy\\[-8pt]}}
\email{carlo.sanna@polito.it}


\title{Explicit inequalities for the $n$th lucky number}

\usepackage{amsmath}
\usepackage{amssymb}
\usepackage{amsthm}
\usepackage[left=1.15in, right=1.15in, top=.72in, bottom=.72in]{geometry}
\usepackage[colorlinks=true]{hyperref}
\usepackage{cleveref}
\usepackage{comment}
\usepackage{enumitem}
\setlist[enumerate]{label=(\roman*),labelindent=1em,itemsep=0.5em,topsep=0.5em}

\newtheorem{theorem}{Theorem}[section]

\newtheorem{lemma}[theorem]{Lemma}
\theoremstyle{remark}
\newtheorem{remark}{Remark}[section]

\newcommand{\neper}{\mathrm{e}}
\DeclareMathOperator{\llog}{llog}

\allowdisplaybreaks
\newcommand{\ConMaxLuckyIndex}{10^7}
\newcommand{\ConFlbNzer}{549}
\newcommand{\ConFlbNone}{66593}
\newcommand{\ConRoneNone}{66621}
\newcommand{\ConRoneNtwo}{739945}
\newcommand{\ConRoneNthr}{9999862}
\newcommand{\ConRoneCone}{0.991755851885?}
\newcommand{\ConRoneCtwo}{1.268476992459?}
\newcommand{\ConRoneCthr}{0.5752440606595?}
\newcommand{\ConRoneCfou}{3.6024414?}
\newcommand{\ConRoneCfiv}{0.008278319041?}
\newcommand{\ConRoneCsix}{1.95351805?}
\newcommand{\ConRoneCsev}{1.016590998114?}
\newcommand{\ConRoneCeig}{7.246544?}
\newcommand{\ConRtwoNone}{66621}
\newcommand{\ConRtwoNtwo}{739945}
\newcommand{\ConRtwoNthr}{9999862}

\newcommand{\ConRtwoCtwo}{1.27675531149927?}
\newcommand{\ConRtwoCthr}{0.5752440606595?}
\newcommand{\ConRtwoCfou}{3.6146231?}
\newcommand{\ConRtwoCfiv}{0}
\newcommand{\ConRtwoCsix}{1.94111497?}
\newcommand{\ConRtwoCsev}{1}
\newcommand{\ConRtwoCeig}{7.206931?}
\newcommand{\ConRthrNone}{739954}
\newcommand{\ConRthrNtwo}{9999993}
\newcommand{\ConRthrCtwo}{1.23864439772699?}
\newcommand{\ConRthrCthr}{0.5661305487495?}
\newcommand{\ConRthrCfou}{3.015515?}
\newcommand{\ConLowCeig}{7.207}
\newcommand{\ConLowDE}{4.927}
\newcommand{\ConUppCtwo}{1.239}
\newcommand{\ConUppCfou}{3.016}

\begin{document}

\begin{abstract}
    Gardiner, Lazarus, Metropolis, and Ulam introduced a variation of the sieve of Eratosthenes that---instead of producing the sequence of prime numbers---produces the sequence of \emph{lucky numbers}.
    The distribution of lucky numbers has a striking similarity to that of prime numbers.
    In particular, Hawkins and Briggs proved that if $\ell_n$ denotes the $n$th lucky number then $\ell_n \sim n \log n$, which is analogous to the prime number theorem.

    This work provides explicit upper and lower bounds on $\ell_n$.
\end{abstract}

\maketitle

\section{Introduction}

In 1956, Gardiner, Lazarus, Metropolis, and Ulam~\cite{MR75217} introduced the following variation of the sieve of Eratosthenes, which---instead of producing the sequence of prime numbers---produces a sequence of integers that they called \emph{lucky numbers}.
Start with the sequence of positive integers
\begin{equation*}
    1,\; 2,\; 3,\; 4,\; 5,\; 6,\; 7,\; 8,\; 9,\; 10,\; \dots
\end{equation*}
Remove every \emph{second} term.
This leaves the sequence of odd positive integers
\begin{equation*}
    1,\; 3,\; 5,\; 7,\; 9,\; 11,\; 13,\; 15,\; 17,\; 19,\;  \dots
\end{equation*}
Except for $1$, the first remaining integer is $3$.
Remove every \emph{third} term, leaving the sequence
\begin{equation*}
    1,\; 3,\; 7,\; 9,\; 13,\; 15,\; 19,\; 21,\; 25,\; 27,\; \dots
\end{equation*}
The first new remaining integer is $7$.
Remove every \emph{seventh} term... and so on.
Continuing this process indefinitely leaves the sequence of lucky numbers
\begin{equation*}
    1,\; 3,\; 7,\; 9,\; 13,\; 15,\; 21,\; 25,\; 31,\; 33,\; \dots
\end{equation*}
See the OEIS for more terms~\cite[A000959]{OEIS}.
By numerical experiments, Gardiner et al.\ observed a striking similarity between the distribution of lucky numbers and that of prime numbers.
Let $\ell_n$ and $p_n$ denote the $n$th lucky number and the $n$th prime number, respectively.
In 1957, Hawkins and Briggs~\cite{MR103866,MR103867} proved that $\ell_n \sim n \log n$, which is analogous to the prime number theorem $p_n \sim n \log n$.
They also claimed that Chowla applied their method recursively and proved that
\begin{equation*}
    \ell_n
    = n \big(\! \log n + \big(\tfrac1{2} + o(1)\big) (\log \log n)^2 \big) .
\end{equation*}
Since the corresponding result for prime numbers is
\begin{equation*}
    p_n
    = n \big(\! \log n + \big(1 + o(1)\big) \log \log n \big) ,
\end{equation*}
it follows that $\ell_n > p_n$ for all sufficiently large integers $n$.
In 1963, Briggs~\cite{MR148638} provided asymptotic formulas for a large class of sequences, which contains the sequence of lucky numbers.
In~particular, he proved that
\begin{equation}\label{equ:l_n-third-term}
    \ell_n
    = n \big(\! \log n + \tfrac1{2} (\log \log n)^2 - (\gamma + o(1)) (\log \log n)\big) ,
\end{equation}
where $\gamma = 0.577...$ is the Euler--Mascheroni constant.
It is worth to remark that in 1958 Erd\H{o}s and Jabotinsky~\cite{MR103865} independently proved a result that is substantially the same as~\eqref{equ:l_n-third-term}.

The aim of this work is to prove explicit inequalities for the $n$th lucky number.
The first result is a lower bound on $\ell_n$.

\begin{theorem}\label{thm:lower}
    The inequality
    \begin{equation}\label{thm:lower:equ:1}
        \ell_n > n \log n
    \end{equation}
    holds for all integers $n \geq 1$; and the inequality
    \begin{equation}\label{thm:lower:equ:2}
        \ell_n
        > n \big(\!\log n + \tfrac1{2}(\log\log n)^2 - \log\log n - \ConLowCeig\big)
    \end{equation}
    holds for all integers $n > \neper^{\neper^{\ConLowDE}}$.
\end{theorem}

Note that \eqref{thm:lower:equ:1} is analogous to Rosser's theorem~\cite{MR1576808}, which states that $p_n > n \log n$ for all integers $n \geq 1$.
Note also that the condition $n > \neper^{\neper^{\ConLowDE}}$ ensures that, for such integers $n$, lower bound \eqref{thm:lower:equ:2} is stronger than \eqref{thm:lower:equ:1}.
Another observation is that inequality \eqref{thm:lower:equ:2} is quite close to the expression in asymptotic formula \eqref{equ:l_n-third-term}, since it has the same two main terms and a third term that (asymptotically) differs only by a factor of $1/\gamma = 1.73...$

The second result is an upper bound on $\ell_n$.

\begin{theorem}\label{thm:upper}
    The inequality
    \begin{equation}\label{thm:upper:equ:1}
        \ell_n
        < n \big(\!\log n + \tfrac1{2}(\log\log n)^2 + 1\big)
    \end{equation}
    holds for all integers $n \in \big[4, 10^7\big]$; and the inequality
    \begin{equation}\label{thm:upper:equ:2}
        \ell_n
        < n \big(\!\log n + \tfrac1{2}(\log\log n)^2 + \ConUppCtwo\log\log n + \ConUppCfou\big)
    \end{equation}
    holds for all integers $n > 10^7$.
\end{theorem}

Note that, in light of asymptotic formula \eqref{equ:l_n-third-term}, upper bound \eqref{thm:upper:equ:2} is comparatively weaker than lower bound \eqref{thm:lower:equ:2}.
Indeed, it seems likely that \eqref{thm:upper:equ:1} is true for all integers $n \geq 4$.

The main strategy of the proof of Theorems~\ref{thm:lower} and~\ref{thm:upper}  mostly follows the ideas of Hawkins and Briggs~\cite{MR103866} (and their hint at Chowla's recursive method).
Indeed, Section~\ref{sec:preliminaries-lucky-numbers} replicates some of their arguments for the sake of completeness/notation.
The novelty is the way in which all error terms are made explicit and all inequalities are proved to hold over large ranges.

The proof requires some numerical computations.
The corresponding source code is available in a public repository~\cite{Repo} (more details in Section~\ref{sec:numerical-computations}).

The structure of this work is the following.
Section~\ref{sec:notation} provides the general notation.
Section~\ref{sec:preliminaries-lucky-numbers} collects some preliminary results on lucky numbers.
Section~\ref{sec:layout} sketches a layout of the proof of Theorems~\ref{thm:lower} and~\ref{thm:upper}.
Sections~\ref{sec:ell_n-first-lower-bound} through \ref{sec:ell_n-bounds} contains the proof with the exception of numerical computations, which are provided in Section~\ref{sec:numerical-computations}.


\section{Notation}\label{sec:notation}

If $x$ is a real number, then $\lfloor x \rfloor$ and $\lceil x \rceil$ denote the maximal integer less than or equal to $x$ and the minimal integer greater than or equal to $x$, respectively; and $\{x\} := x - \lfloor x \rfloor$.
The symbols $\neper$ and $\gamma$ denote the Napier constant $2.718...$ and the Euler--Mascheroni constant $0.577...$, respectively.
The notation for the natural logarithm is $\log$, while $\llog$ is a shorthand for $\log \log$.
The notation $f = \Theta\big(g\big)$ means that $0 \leq f \leq g$.
If $\mathcal{S}$ is a finite set, then $|\mathcal{S}|$ denotes the cardinality of $\mathcal{S}$.
By convention, the empty sum is equal to $0$ and the empty product is equal to $1$.


\section{Preliminaries on lucky numbers}\label{sec:preliminaries-lucky-numbers}

\subsection{Definition of lucky numbers}

This subsection provides a formal definition of the sequence of lucky numbers and introduces some related notation.
It is convenient to define the first lucky number to be $\ell_1 := 2$ (instead of $1$, as in the introduction).
Let $(\mathcal{L}_m)_{m \geq 1}$ be the sequence of sets of positive integers recursively defined as follows.
First, the set $\mathcal{L}_1$ consists of all positive integers and the set $\mathcal{L}_2$ consists of the number $2$ and all odd integers greater than or equal to $3$.
For every integer $m \geq 1$, let $(\ell_{m,n})_{n \geq 1}$ denote the sequence formed by the elements of $\mathcal{L}_m$ in increasing order.
Then, for every integer $m \geq 2$, the set $\mathcal{L}_{m+1}$ consists of all integers $\ell_{m,n}$ such that $n$ is a positive integer not divisible by $\ell_{m,m}$.
For every integer $n \geq 2$, the $n$th lucky number is $\ell_n := \ell_{n,n}$.

\subsection{Two key results}

For all integers $m \geq 2$, $n \geq 1$, and for all real numbers $x$, define
\begin{equation*}
	L_n(x) := \big|\mathcal{L}_n \cap [1, x]\big| ,
	\quad
	\rho_m := \prod_{i \,=\, 1}^{m - 1} \frac1{1 - 1/\ell_i} ,
	\quad
	\tau_{m, n} := \frac1{n} \sum_{i \,=\, 1}^{m - 1} \frac{\rho_{i+1}}{\rho_m} \left\{\frac{L_i(\ell_{m,n})}{\ell_i}\right\} ,
\end{equation*}
and $\tau_m := \tau_{m, m}$.
The following lemma provides a fundamental formula for $\ell_{m,n}$.

\begin{lemma}\label{lem:l_mn-formula}
	Let $m \geq 2$ and $n \geq 1$ be integers.
	Then $\ell_{m, n} = n \rho_m (1 - \tau_{m,n})$.
\end{lemma}
\begin{proof}
	Let $x \geq 2$ be a real number.
	If $m \geq 3$ then the definition of $\mathcal{L}_m$ in terms of $\mathcal{L}_{m-1}$ implies that
	\begin{equation}\label{lem:l_mn-formula:equ:1}
		L_m(x)
		= L_{m - 1}(x) - \left\lfloor \frac{L_{m - 1}(x)}{\ell_{m-1}} \right\rfloor
		= \left(1 - \frac1{\ell_{m-1}}\right) L_{m - 1}(x) + \left\{\frac{L_{m - 1}(x)}{\ell_{m-1}}\right\} .
	\end{equation}
	In addition, it is easy to check that \eqref{lem:l_mn-formula:equ:1} holds also for $m = 2$ (here it is necessary that $x \geq 2$).
	Repeatedly applying \eqref{lem:l_mn-formula:equ:1} and noting that $L_1(x) = \lfloor x \rfloor$ yields
	\begin{equation}\label{lem:l_mn-formula:equ:2}
		L_m(x) = \frac{\lfloor x \rfloor}{\rho_m}  + \sum_{i \,=\, 1}^{m - 1} \frac{\rho_{i+1}}{\rho_m} \left\{\frac{L_i(x)}{\ell_i}\right\} .
	\end{equation}
	The claim follows by plugging $x = \ell_{m,n}$ into \eqref{lem:l_mn-formula:equ:2} and noting that $L_m(\ell_{m,n}) = n$.
\end{proof}

The next lemma provides a sufficient condition for the equality of $\ell_n$ and $\ell_{m,n}$.

\begin{lemma}\label{lem:l_n-equal-l_mn}
	Let $m \geq 2$ and $n \geq 1$ be integers.
	Suppose that $\ell_m > n$.
	Then $\ell_n = \ell_{m, n}$.
\end{lemma}
\begin{proof}
	If $m = n$ then the claim is obvious.
    If $n = 1$ then the claim follows easily.
	Hence, assume that $m \neq n$ and $n \geq 2$.
	Let $k \geq 2$ be an integer.
	From the definition of $\mathcal{L}_{k + 1}$ it follows that
	\begin{equation}\label{lem:l_n-equal-l_mn:equ:1}
		\ell_k > n \;\Rightarrow\; \ell_{k, n} = \ell_{k + 1, n} .
	\end{equation}
	If $m < n$ then applying \eqref{lem:l_n-equal-l_mn:equ:1} for $k = m, m + 1, \dots, n - 1$ yields
	\begin{equation*}
		\ell_{m, n} = \ell_{m + 1, n} = \cdots = \ell_{n, n} = \ell_n ,
	\end{equation*}
	as desired.
	If $m > n$ then noting that $\ell_n > n$ and applying \eqref{lem:l_n-equal-l_mn:equ:1} for $k = n, n + 1, \dots, m - 1$ yields
	\begin{equation*}
		\ell_n = \ell_{n, n} = \ell_{n + 1, n} = \cdots = \ell_{m, n} ,
	\end{equation*}
	as desired.
\end{proof}


\section{Layout of the proof of Theorems~\ref{thm:lower} and~\ref{thm:upper}}\label{sec:layout}

This section sketch a layout of the proof of Theorems~\ref{thm:lower} and~\ref{thm:upper}.
The letters $c_1, c_2, \dots$ and $n_1, n_2, \dots$ denote explicit positive constants.

\subsection{First lower bound}\label{subsec:layout-first-lower-bound}

The proof begins by showing that
\begin{equation}\label{equ:ell_n-first-lower-bound}
	\ell_n > c_1 n \log n
\end{equation}
for all integers $n \geq n_1$ and a constant $c_1 < 1$ (Lemma~\ref{lem:ell_n-first-lower-bound}).

\subsection{First round}\label{subsec:layout-first-round}

The proof continues by showing that \eqref{equ:ell_n-first-lower-bound} implies that
\begin{equation}\label{equ:tau_n-upper-bound}
	\tau_n < \frac1{c_1} \!\left(\frac{\llog n + c_2}{\log n} + c_3 \!\left(\frac{\llog n}{\log n}\right)^{\!2} \right)
\end{equation}
for all integers $n \geq n_2$ (Lemma~\ref{lem:tau_n-upper-bound}).
Then the proof proceeds by showing that \eqref{equ:ell_n-first-lower-bound} and \eqref{equ:tau_n-upper-bound} imply that
\begin{equation}\label{equ:ell_n-upper-bound}
	\ell_n < n \!\left(\log n + \frac{\tfrac1{2} (\llog n)^2 + c_2 \llog n + c_4}{c_1}\right)
\end{equation}
for all integers $n \geq n_2$ (Lemma~\ref{lem:ell_n-upper-bound}).
As a next step, the proof shows that \eqref{equ:ell_n-first-lower-bound} and \eqref{equ:ell_n-upper-bound} imply that
\begin{equation}\label{equ:tau_n-lower-bound}
	\tau_n > \frac{\llog n - c_5}{\log n} - \frac{c_6(\llog n)^3}{(\log n)^2}
\end{equation}
for all integers $n \geq n_3$ (Lemma~\ref{lem:tau_n-lower-bound}).
The proof continues by showing that \eqref{equ:tau_n-upper-bound} and \eqref{equ:tau_n-lower-bound} imply that
\begin{equation}\label{equ:ell_n-second-lower-bound}
	\ell_n > n \big(\!\log n + \tfrac1{2}(\llog n)^2 - c_7 \llog n - c_8\big)
\end{equation}
for all integers $n \geq n_3$ (Lemma~\ref{lem:ell_n-second-lower-bound}).

\subsection{Bootstrapping}

From \eqref{equ:ell_n-second-lower-bound} it follows that \eqref{thm:lower:equ:1}, the first lower bound of Theorem~\ref{thm:lower}, is true for all sufficiently large integers $n$, say $n \geq n_4$.
The constant $n_4$ is quite large, but an ad~hoc reasoning (Lemma~\ref{lem:ell_n-finite-range}) shows that \eqref{thm:lower:equ:1} holds also for the positive integers less than $n_4$.
This proves that \eqref{thm:lower:equ:1} is true for all integers $n \geq 1$, as desired.

\subsection{Second round}

In light of \eqref{thm:lower:equ:1} holding for all integers $n \geq 1$, the proof repeats the reasoning of Subsection~\ref{subsec:layout-first-round} setting $c_1 = 1$ from the beginning.
This proves inequality \eqref{thm:lower:equ:2} and completes the proof of Theorem~\ref{thm:lower}.

\subsection{Third half-round}

At this stage, the proof already gave inequality \eqref{equ:ell_n-upper-bound}, which is an upper bound on $\ell_n$ of the same form of \eqref{thm:upper:equ:2} but with slightly worse constants.
Only to obtain better constants, the proof repeats the reasoning of Subsection~\ref{subsec:layout-first-round} up to inequality \eqref{equ:ell_n-upper-bound} setting $c_1 = 1$ and employing a larger $n_1$.
This yields \eqref{thm:upper:equ:2}.
Finally, a direct computation proves \eqref{thm:upper:equ:1} and completes the proof of Theorem~\ref{thm:upper}.


\section{First lower bounds on \texorpdfstring{$\ell_n$}{ell\_n}}\label{sec:ell_n-first-lower-bound}

This section is devoted to prove some lower bounds on $\ell_n$.
For every integer $n \geq 2$, define
\begin{equation}\label{equ:varrho_n-definition}
	\varrho_n := \rho_n - 1 - \sum_{k \,=\, 1}^{n - 1} \frac1{k} .
\end{equation}
The next lemma collects some properties of $\varrho_n$.

\begin{lemma}\label{lem:varrho_n-basics}
	Let $n$ and $n_0$ be integers such that $n \geq n_0 \geq 2$.
	Then
	\begin{align}
		\rho_n
		&= \log n + \varrho_n + \gamma + 1 - \Theta\!\left(\frac{0.542}{n}\right) , \label{lem:varrho_n-basics:equ:1} \\
		\varrho_n
		&= \sum_{k \,=\, 2}^{n - 1} \frac1{k}\!\left(\frac1{(1 - 1/\ell_k)(1 - \tau_k)} - 1\right) , \label{lem:varrho_n-basics:equ:2} \\
		\varrho_n &\geq \varrho_{n_0} + \sum_{k \,=\, n_0}^{n - 1} \frac{\tau_k}{k} \label{lem:varrho_n-basics:equ:3} .
	\end{align}
    Furthermore, the quantity $\varrho_n$ is an increasing function of $n$.
\end{lemma}
\begin{proof}
    A well-known asymptotic for the $n$th harmonic number (see, e.g., \cite[Theorem~0.8]{MR3363366}) states that
	\begin{equation}\label{lem:varrho_n-basics:equ:4}
		\sum_{k \,=\, 1}^n \frac1{k}
		= \log n + \gamma + \frac1{2n} - \Theta\!\left(\frac1{12n^2}\right) .
	\end{equation}
    Then \eqref{equ:varrho_n-definition} and \eqref{lem:varrho_n-basics:equ:4} imply that
    \begin{equation*}
        \rho_n
        = \varrho_n + 1 + \sum_{k \,=\, 1}^{n-1} \frac1{k}
        = \varrho_n + 1 - \frac1{n} + \sum_{k \,=\, 1}^n \frac1{k}
        = \log n + \varrho_n + \gamma + 1 - \frac1{2n} - \Theta\!\left(\frac1{12n^2}\right) ,
    \end{equation*}
    which in turn gives \eqref{lem:varrho_n-basics:equ:1}.
	Let $k \geq 2$ be an integer.
	The definition of $\rho_k$ implies that
	\begin{equation}\label{lem:varrho_n-basics:equ:5}
		\frac{\rho_k}{\rho_{k+1}}
		= 1 - \frac1{\ell_k} .
	\end{equation}
	From \eqref{lem:varrho_n-basics:equ:5} and Lemma~\ref{lem:l_mn-formula} it follows that
	\begin{equation*}
		\rho_{k+1} - \rho_k
		= \rho_{k+1} \left(1 - \frac{\rho_k}{\rho_{k+1}}\right)
		= \frac{\rho_{k+1}}{\ell_k}
		= \frac{\rho_{k+1}}{k \rho_k (1 - \tau_k)}
		= \frac1{k (1 - 1/\ell_k) (1 - \tau_k)}
	\end{equation*}
	and so
	\begin{equation}\label{lem:varrho_n-basics:equ:6}
		\rho_n
		= \rho_2 + \sum_{k \,=\, 2}^{n - 1} (\rho_{k + 1} - \rho_k)
		= 2 + \sum_{k \,=\, 2}^{n - 1} \frac1{k (1 - 1/\ell_k) (1 - \tau_k)} .
	\end{equation}
	Then \eqref{equ:varrho_n-definition} and \eqref{lem:varrho_n-basics:equ:6} imply \eqref{lem:varrho_n-basics:equ:2}.
	Since $\ell_k > 1$ and $\tau_k < 1$, it follows that
    \begin{equation*}
        \frac1{(1 - 1/\ell_k)(1 - \tau_k)}
        > \frac1{1 - \tau_k}
        \geq 1 + \tau_k ,
    \end{equation*}
    which in tandem with \eqref{lem:varrho_n-basics:equ:2} yields \eqref{lem:varrho_n-basics:equ:3}.
    Finally, from \eqref{lem:varrho_n-basics:equ:3} and the fact that $\tau_k \geq 0$ it follows that $\varrho_n$ is an increasing function of $n$.
\end{proof}

The following lemma provides a first lower bound on $\ell_n$.

\begin{lemma}\label{lem:ell_n-first-lower-bound}
	Let $n_0 \geq 2$ be an integer and define
	\begin{align}
		c_0 &:= \varrho_{n_0} + \gamma + 1 - \frac{0.542}{n_0} , \label{equ:constant-c0} \\
		c_1 &:= 1 - \neper^{-c_0} , \nonumber \\
		n_1 &:= \lceil \neper^{c_0} n_0 \rceil . \nonumber
	\end{align}
	Then \eqref{equ:ell_n-first-lower-bound} holds for all integers $n \geq n_1$.
\end{lemma}
\begin{proof}
	Let $m \geq n_0$ be an integer.
    Lemma~\ref{lem:varrho_n-basics} states that $\varrho_n$ is an increasing function of $n$.
    This fact and Eq.\ \eqref{lem:varrho_n-basics:equ:1} of Lemma~\ref{lem:varrho_n-basics} imply that
	\begin{equation}\label{lem:ell_n-first-lower-bound:equ:3}
		\rho_m
		\geq \log m + \varrho_m + \gamma + 1 - \frac{0.542}{m}
		\geq \log m + \varrho_{n_0} + \gamma + 1 - \frac{0.542}{n_0}
		= \log m + c_0 .
	\end{equation}
    Let $n \geq 1$ be an integer.
    Lemma~\ref{lem:l_mn-formula}, \eqref{lem:ell_n-first-lower-bound:equ:3}, and the fact that $\tau_{m,n} \leq (m - 1)/n$ imply that
    \begin{equation}\label{lem:ell_n-first-lower-bound:equ:4}
        \ell_n
        \geq \ell_{m,n}
        = n \rho_m (1 - \tau_{m,n})
        \geq n (\log m + c_0) \!\left(1 - \frac{m - 1}{n}\right) .
    \end{equation}
    Suppose that $n \geq n_1$ and pick $m := \lceil \neper^{-c_0} n \rceil$.
    Hence $m \geq n_0$ so that \eqref{lem:ell_n-first-lower-bound:equ:4} holds.
    From \eqref{lem:ell_n-first-lower-bound:equ:4} and the inequalities $\neper^{-c_0} n \leq m < \neper^{-c_0} n + 1$ it follows that
    \begin{equation*}
        \ell_n
        > n (\log n - c_0 + c_0) (1 - \neper^{-c_0})
        = c_1 n \log n ,
    \end{equation*}
    which is \eqref{equ:ell_n-first-lower-bound}, as desired.
\end{proof}

The next lemma shows that \eqref{thm:lower:equ:1} holds over a finite, but large, range of integers.

\begin{lemma}\label{lem:ell_n-finite-range}
    Let $n_0 \geq 2$ be an integer, let $c_0$ be given by \eqref{equ:constant-c0}, let $t \leq 1$ be a real number, and define
    \begin{align*}
        m_1 &:= \left\lceil \neper^{c_0 t} n_0 \right\rceil , \\
        m_2 &:= \left\lfloor \exp\!\left(c_0(1 - t)(\neper^{c_0 t} - 1)\right) \right\rfloor .
    \end{align*}
    Then \eqref{thm:lower:equ:1} holds for all integers $n \in [m_1, m_2]$.
\end{lemma}
\begin{proof}
    Continue from the end of the proof of Lemma~\ref{lem:ell_n-first-lower-bound}.
    Let $n \in [m_1, m_2]$ be an integer and pick $m := \lceil \neper^{-c_0 t} n \rceil$.
    Since $n \geq \neper^{c_0 t} n_0$, it follows that $m \geq n_0$ and so \eqref{lem:ell_n-first-lower-bound:equ:4} holds.
    Moreover, from \eqref{lem:ell_n-first-lower-bound:equ:4} and the inequalities
    \begin{equation*}
        \neper^{-c_0 t} n \leq m < \neper^{-c_0 t} n + 1
    \end{equation*}
    and
    \begin{equation*}
        n \leq \exp\!\left(c_0(1 - t)(\neper^{c_0 t} - 1)\right) ,
    \end{equation*}
    it follows that
    \begin{align*}
        \ell_n
        > n \big(\!\log n + c_0(1 - t)\big) \big(1 - \neper^{-c_0 t}\big)
        = \left(1 + \frac{c_0 (1 - t)}{\log n}\right) \!\big(1 - \neper^{-c_0 t}\big) \, n \log n
        \geq n \log n ,
    \end{align*}
    which is \eqref{thm:lower:equ:1}, as desired.
\end{proof}


\section{Interlude: Some technical lemmas}

This section collects some technical results that are needed in later proofs.

\subsection{Some estimates}

The next lemma provides an estimate for a certain sum of products.

\begin{lemma}\label{lem:prod-vs-sum}
	Let $x_1, \dots, x_n \in (0, 1)$.
	Then
	\begin{equation}\label{lem:prod-vs-sum:equ:1}
		\sum_{i \,=\, 1}^n x_i \!\!\prod_{j \,=\, i + 1}^n (1 - x_j)
		= \sum_{i \,=\, 1}^n x_i - \Theta\Big(\tfrac1{2}\Big(\sum_{i \,=\, 1}^n x_i\Big)^2\Big) .
	\end{equation}
\end{lemma}
\begin{proof}
	The inequality
	\begin{equation}\label{lem:prod-vs-sum:equ:2}
		\sum_{i \,=\, 1}^n x_i \!\!\prod_{j \,=\, i + 1}^n (1 - x_j)
		\leq \sum_{i \,=\, 1}^n x_i
	\end{equation}
	is straightforward.
	It follows easily by induction on $n$ that
	\begin{equation}\label{lem:prod-vs-sum:equ:3}
		\sum_{i \,=\, 1}^n x_i \!\!\prod_{j \,=\, i + 1}^n (1 - x_j)
		= 1 - \prod_{i \,=\, 1}^n (1 - x_i) .
	\end{equation}
	Since $\neper^{-x} = 1 - x + \Theta(x^2 \!/ 2)$ for every real number $x \geq 0$, it follows that
	\begin{equation}\label{lem:prod-vs-sum:equ:4}
		1 - \prod_{i \,=\, 1}^n (1 - x_i)
		\geq 1 - \prod_{i \,=\, 1}^n \neper^{-x_i}
		= 1 - \exp\!\Big(\!-\sum_{i \,=\, 1}^n x_i\Big)
		\geq \sum_{i \,=\, 1}^n x_i - \tfrac1{2}\Big(\sum_{i \,=\, 1}^n x_i\Big)^2 .
	\end{equation}
	Putting together \eqref{lem:prod-vs-sum:equ:2}, \eqref{lem:prod-vs-sum:equ:3}, and \eqref{lem:prod-vs-sum:equ:4} yields \eqref{lem:prod-vs-sum:equ:1}.
\end{proof}

The following lemma is a well-known estimate of a sum by an integral.
The proof is given just for completeness.

\begin{lemma}\label{lem:sum-to-integral}
	Let $a$ and $b$ be real numbers such that $a < b$ and let $f \colon [a, b] \to {[0, +\infty)}$ be a monotone decreasing function with continuous derivative.
	Then
	\begin{equation}\label{lem:sum-to-integral:equ:2}
		\sum_{a \,<\, n \,<\, b} f(n)
		= \int_a^b f(t)\,\mathrm{d}t \pm \Theta\big(f(a)\big) .
	\end{equation}
\end{lemma}
\begin{proof}
	From Euler--MacLaurin formula \cite[Proposition~1.2]{MR2919246} it follows that
	\begin{equation}\label{lem:sum-to-integral:equ:3}
		\sum_{a \,<\, n \,\leq\, b} f(n)
		= \int_a^b f(t)\,\mathrm{d}t + \{a\}f(a) - \{b\}f(b) + \int_a^b \{t\}f^\prime(t)\,\mathrm{d}t .
	\end{equation}
	Since $f$ is monotone decreasing, the inequality $f^\prime(t) \leq 0$ holds for every $t \in [a, b]$.
	Hence
	\begin{equation}\label{lem:sum-to-integral:equ:4}
		0 \leq -\int_a^b \{t\}f^\prime(t)\,\mathrm{d}t \leq -\int_a^b f^\prime(t)\,\mathrm{d}t = f(a) - f(b) .
	\end{equation}
	Putting together \eqref{lem:sum-to-integral:equ:3} and \eqref{lem:sum-to-integral:equ:4} yields
	\begin{equation}\label{lem:sum-to-integral:equ:1}
		\sum_{a \,<\, n \,\leq\, b} f(n)
		= \int_a^b f(t)\,\mathrm{d}t + \{a\}f(a) - \{b\}f(b) - \Theta\big(f(a) - f(b)\big) .
	\end{equation}
	If $b$ is an integer, then \eqref{lem:sum-to-integral:equ:1} and the fact that $f$ is monotone decreasing and nonnegative imply that
	\begin{align*}
		\sum_{a \,<\, n \,<\, b} f(n)
		&= \Big(\sum_{a \,<\, n \,\leq\, b} f(n) \Big) - f(b)
		= \int_a^b f(t)\,\mathrm{d}t + \{a\}f(a) - \Theta\big(f(a) - f(b)\big) - f(b) \\
		&= \int_a^b f(t)\,\mathrm{d}t \pm \Theta\big(f(a)\big) ,
	\end{align*}
	which is \eqref{lem:sum-to-integral:equ:2}.
	If $b$ is not an integer, then \eqref{lem:sum-to-integral:equ:1} and the fact that $f$ is monotone decreasing and nonnegative imply that
	\begin{align*}
		\sum_{a \,<\, n \,<\, b} f(n)
		&= \sum_{a \,<\, n \,\leq\, b} f(n)
		= \int_a^b f(t)\,\mathrm{d}t + \{a\}f(a) - \{b\}f(b) - \Theta\big(f(a) - f(b)\big) \\
		&= \int_a^b f(t)\,\mathrm{d}t \pm \Theta\big(f(a)\big) ,
	\end{align*}
	which is \eqref{lem:sum-to-integral:equ:2} again.
\end{proof}

\subsection{Logarithms}

The following lemma regards the monotonicity of a function involving the natural logarithm.

\begin{lemma}\label{lem:log-t-to-a-over-t}
	Let $a \geq 0$ be a real number.
	Then the function $x \mapsto (\log x)^a \!/x$ is monotone increasing over $[1, \neper^a]$ and monotone decreasing over ${[\neper^a, +\infty)}$.
\end{lemma}
\begin{proof}
	The claim follows easily by considering the sign of the first derivative.
\end{proof}

The next lemma is a technical inequality.

\begin{lemma}\label{lem:tricky-inequality}
	Let $x$ and $y$ be real numbers such that $x \geq 11.51$ and
	\begin{equation}\label{lem:tricky-inequality:equ:1}
		\frac{\log\!\big(0.99(x - \log x)\big)}{x} \leq y < 1 .
	\end{equation}
	Then
	\begin{equation}\label{lem:tricky-inequality:equ:2}
		-\log(1 - y) - y
		> \frac{-\log\!\big(1 - (\log x) / x\big)}{x} + \neper^{-x} .
	\end{equation}
\end{lemma}
\begin{proof}
	Let $x_0$ and $x_1$ be real numbers such that $x_1 \geq x_0 \geq \neper$, and define
	\begin{equation*}
		y_0 := \frac{\log\!\big(0.99 (x_0 - \log x_0)\big)}{x_1} , \quad
		a := \frac{-\log(1 - y_0) - y_0}{y_0^2} .
	\end{equation*}
	Furthermore, extend these definitions to $x_1 = +\infty$ by setting $y_0 := 0$ and $a := 1/2$.

	Suppose that $x \in {[x_0, x_1)}$.
	From \eqref{lem:tricky-inequality:equ:1} it follows that $y \geq y_0$, which in turn---using \eqref{lem:tricky-inequality:equ:1} again---implies that
	\begin{equation}\label{lem:tricky-inequality:equ:3}
		-\log(1 - y) - y
		\geq a y^2
		\geq a \!\left(\frac{\log\!\big(0.99(x - \log x)\big)}{x}\right)^{\!2} .
	\end{equation}
	At this point, note that the inequality
	\begin{equation}\label{lem:tricky-inequality:equ:4}
		a \!\left(\frac{\log\!\big(0.99(x - \log x)\big)}{x}\right)^{\!2}
		> \frac{-\log\!\big(1 - (\log x) / x\big)}{x} + \neper^{-x}
	\end{equation}
	is equivalent to
	\begin{equation}\label{lem:tricky-inequality:equ:5}
		a \!\left(1 - \frac{\log(1/0.99) - \log\!\big(1 - (\log x) / x\big)}{\log x}\right)^{\!2} \log x
		> \frac{-\log\!\big(1 - (\log x)/x\big)}{(\log x) / x} + \frac{x^2}{\neper^x \log x} .
	\end{equation}
	Furthermore, by Lemma~\ref{lem:log-t-to-a-over-t}, the left-hand, resp.\ right-hand, side of \eqref{lem:tricky-inequality:equ:5} is a function of $x$ that is increasing, resp.\ decreasing, over $[\neper, +\infty)$.
	Hence, if \eqref{lem:tricky-inequality:equ:5} holds for $x = x_0$ then it holds for all real numbers $x \in {[x_0, x_1)}$, and so does \eqref{lem:tricky-inequality:equ:4}.
	Then inequalities \eqref{lem:tricky-inequality:equ:3} and \eqref{lem:tricky-inequality:equ:4} imply that \eqref{lem:tricky-inequality:equ:2} is true for all real numbers $x \in {[x_0, x_1)}$.

	Finally, a computation shows that if $(x_0, x_1)$ is equal to $(11.51, 12)$, $(12, 14)$, or $(14, +\infty)$ then \eqref{lem:tricky-inequality:equ:5} holds for $x = x_0$.
	Therefore, inequality \eqref{lem:tricky-inequality:equ:2} is true for all real numbers $x \geq 11.51$, as desired.
\end{proof}

\subsection{Lambert W function}\label{subsec:lambert-w-function}

Let $W$ denote the principal branch of the Lambert W function.
Recall that $W$ is the unique function from ${[-1/\neper, +\infty)}$ to ${[-1, +\infty)}$ that satisfies
\begin{equation}\label{equ:lambert-w-definition}
	W(x) \neper^{W(x)} = x
\end{equation}
for every real number $x \geq -1/\neper$.
(For more on the Lambert W function, see the article by Corless, Gonnet, Hare, Jeffrey, and Knuth~\cite{MR1414285} and the book by Mez\H{o}~\cite{MR4600791}.)
Furthermore, define
\begin{align}\label{equ:omega-definition}
	\omega(x) &:= \neper^{W(x)}
\end{align}
for every real number $x \geq -1/\neper$.

The following lemma collects some properties of $\omega(x)$.

\begin{lemma}\label{lem:omega-properties}
	Let $x \geq \neper$ be a real number.
	Then
	\begin{align}
		\omega(x) \log \omega(x) &= x , \label{lem:omega-properties:equ:1} \\
		\omega(x \log x) &= x , \label{lem:omega-properties:equ:2} \\
		\omega(x) &\leq \frac{x}{\log x - \llog x} , \label{lem:omega-properties:equ:3} \\
		\log \omega(x) &= \log x - \llog x + \Theta\!\left(\!-\log\!\left(1 - \frac{\llog x}{\log x}\right)\!\right) . \label{lem:omega-properties:equ:4}
	\end{align}
\end{lemma}
\begin{proof}
	Equations \eqref{lem:omega-properties:equ:1} and \eqref{lem:omega-properties:equ:2} follow quickly from \eqref{equ:lambert-w-definition} and \eqref{equ:omega-definition}.
	By taking the logarithms of both sides of \eqref{equ:lambert-w-definition} and rearranging, it follows that
	\begin{equation}\label{lem:omega-properties:equ:5}
		W(x) = \log x - \log W(x) .
	\end{equation}
	From $x \geq \neper$ and \eqref{equ:lambert-w-definition} it follows that $W(x) \geq 1$, which in tandem with \eqref{lem:omega-properties:equ:5} implies that
	\begin{equation}\label{lem:omega-properties:equ:6}
		W(x) \leq \log x .
	\end{equation}
	In turn, from \eqref{lem:omega-properties:equ:5} and \eqref{lem:omega-properties:equ:6} it follows that
	\begin{equation}\label{lem:omega-properties:equ:7}
		W(x) \geq \log x - \llog x .
	\end{equation}
	Then \eqref{equ:omega-definition}, \eqref{equ:lambert-w-definition}, and \eqref{lem:omega-properties:equ:7} imply that
	\begin{equation*}
		\omega(x)
		= \neper^{W(x)}
		= \frac{x}{W(x)}
		\leq \frac{x}{\log x - \llog x} ,
	\end{equation*}
	which is \eqref{lem:omega-properties:equ:3}.
	Moreover, from \eqref{lem:omega-properties:equ:5} and \eqref{lem:omega-properties:equ:7} it follows that
	\begin{equation}\label{lem:omega-properties:equ:8}
		W(x) \leq \log x - \log\!\big(\!\log x - \llog x\big)
		= \log x - \llog x - \log\!\left(1 - \frac{\llog x}{\log x}\right) .
	\end{equation}
	Finally, combining \eqref{equ:omega-definition}, \eqref{lem:omega-properties:equ:7}, and \eqref{lem:omega-properties:equ:8}  yields \eqref{lem:omega-properties:equ:4}.
\end{proof}

The next lemma provides an estimate for a sum involving $\omega(n)$.

\begin{lemma}\label{lem:omega-sum}
	Let $c_1 \in [0.99, 1]$, let $n_2 \geq 10^5$ be an integer, and define
	\begin{equation}\label{equ:constant-c_3}
		c_3 := \frac{-\log\!\big(1 - (\llog n_2)/\!\log n_2\big)- (\llog n_2)/\!\log n_2}{\big((\llog n_2)/\!\log n_2\big)^2} + \frac{(\log n_2)^2}{n_2 (\llog n_2)^2}
	\end{equation}
	Then
	\begin{align}\label{lem:omega-sum:equ:1}
		\sum_{\omega(n) / c_1 \,<\, k \,<\, n} \frac1{k \log k}
		= \frac{\llog n - \log(1/c_1)}{\log n} + \Theta\!\left(\!c_3\!\left(\frac{\llog n}{\log n}\right)^{\!2}\right) .
	\end{align}
	for all integers $n \geq n_2$.
\end{lemma}
\begin{proof}
	Let $n \geq n_2$ be an integer.
	Since $c_1 \in [0.99, 1]$ and $n_2 \geq 10^5$, it follows that
	\begin{equation}\label{lem:omega-sum:equ:1-and-a-half}
		\frac{\omega(n)}{c_1} \geq \omega(n) \geq \omega\big(10^5\big) = 10770.5... > 2
	\end{equation}
	and
	\begin{equation}\label{lem:omega-sum:equ:2}
		c_1 (\log n - \llog n) \geq 0.99 \big(\!\log 10^5 - \llog 10^5\big) = 8.9... > 1 .
	\end{equation}
	Then \eqref{lem:omega-sum:equ:2} and Eq.\ \eqref{lem:omega-properties:equ:3} of Lemma~\ref{lem:omega-properties} imply that
	\begin{equation}\label{lem:omega-sum:equ:3}
		\frac{\omega(n)}{c_1}
		\leq \frac{n}{c_1(\log n - \llog n)}
		< n .
	\end{equation}
	Moreover, since the function $x \mapsto x \log x$ is monotone increasing over $[2, +\infty)$, from \eqref{lem:omega-sum:equ:1-and-a-half} and Eq.\ \eqref{lem:omega-properties:equ:1} of Lemma~\ref{lem:omega-properties} it follows that
	\begin{equation}\label{lem:omega-sum:equ:4}
		\frac{\omega(n)}{c_1} \log\!\left(\frac{\omega(n)}{c_1}\right)
		\geq \omega(n) \log \omega(n)
		= n .
	\end{equation}
	Note that the function $x \mapsto 1/(x \log x)$ is nonnegative and monotone decreasing with continuous derivative over ${[2, +\infty)}$.
	Hence, thanks to \eqref{lem:omega-sum:equ:1-and-a-half} and \eqref{lem:omega-sum:equ:3}, using Lemma~\ref{lem:sum-to-integral} it is possible to estimate the sum in \eqref{lem:omega-sum:equ:1} with the corresponding integral.
	More precisely, using \eqref{lem:omega-sum:equ:4} to bound the $\Theta$-term in \eqref{lem:sum-to-integral:equ:2}, Lemma~\ref{lem:sum-to-integral} implies that
	\begin{equation}\label{lem:omega-sum:equ:5}
		\sum_{\omega(n) / c_1 \,<\, k \,<\, n} \frac1{k \log k}
		= \int_{\omega(n) / c_1}^n \frac{\mathrm{d}t}{t \log t} \pm \Theta\!\left(\frac1{n}\right) .
	\end{equation}
	Computing the integral in \eqref{lem:omega-sum:equ:5} and employing Eq.\ \eqref{lem:omega-properties:equ:4} of Lemma~\ref{lem:omega-properties} gives
	\begin{equation}\label{lem:omega-sum:equ:6}
		\int_{\omega(n) / c_1}^n \frac{\mathrm{d}t}{t \log t}
		= \big[\!\llog t\big]_{t \,=\, \omega(n) / c_1}^n
		= -\log\!\left(\frac{\log\!\big(\omega(n) / c_1\big)}{\log n}\right)
		= -\log(1 - y) ,
	\end{equation}
	where
	\begin{equation}\label{lem:omega-sum:equ:7}
		y := \frac{\llog n - \log(1/c_1) - \Theta\!\left(-\log\!\big(1 - (\llog n)/\!\log n\big)\right)}{\log n} .
	\end{equation}
	From \eqref{lem:omega-sum:equ:7} and $c_1 \leq 1$ it follows that
	\begin{equation}\label{lem:omega-sum:equ:8}
		\frac{\log\!\big(c_1 (\log n - \llog n)\big)}{\log n}
		\leq y
		\leq \frac{\llog n}{\log n} .
	\end{equation}
	In particular, note that $y > 0$ by \eqref{lem:omega-sum:equ:2} and the first inequality in \eqref{lem:omega-sum:equ:8}.
	Putting together \eqref{lem:omega-sum:equ:5}, \eqref{lem:omega-sum:equ:6}, and \eqref{lem:omega-sum:equ:7} yields that
	\begin{equation}\label{lem:omega-sum:equ:9}
		\sum_{\omega(n) / c_1 \,<\, k \,<\, n} \frac1{k \log k}
		= \frac{\llog n - \log(1/c_1)}{\log n} + E
	\end{equation}
	where
	\begin{equation*}
		E := -\log(1 - y) - y
		- \frac{\Theta\!\left(-\log\!\big(1 - (\llog n)/\!\log n\big)\right)}{\log n} \pm \Theta\!\left(\frac1{n}\right) .
	\end{equation*}
	It remains only to estimate the error term $E$.
	On the one hand, since $c_1 \geq 0.99$ and
    \begin{equation*}
        \log n \geq \log n_2 \geq \log 10^5 > 11.51 ,
    \end{equation*}
    the first inequality in \eqref{lem:omega-sum:equ:8} and Lemma~\ref{lem:tricky-inequality} imply that $E \geq 0$.
	On the other hand, since $n \geq n_2 > \neper^\neper$, from the second inequality in \eqref{lem:omega-sum:equ:8} and Lemma~\ref{lem:log-t-to-a-over-t} it follows that
	\begin{equation*}
		E
		\leq -\log(1 - y) - y + \frac1{n}
		\leq \left(\frac{-\log(1 - y) - y}{y^2} + \frac{(\log n)^2}{n (\llog n)^2}\right)\!\left(\frac{\llog n}{\log n}\right)^{\!2}
		\leq c_3 \!\left(\frac{\llog n}{\log n}\right)^{\!2} .
	\end{equation*}
	Hence
	\begin{equation}\label{lem:omega-sum:equ:10}
		E = \Theta\!\left(\!c_3\!\left(\frac{\llog n}{\log n}\right)^{\!2}\right) .
	\end{equation}
	Finally, combining \eqref{lem:omega-sum:equ:9} and \eqref{lem:omega-sum:equ:10} yields \eqref{lem:omega-sum:equ:1}, as desired.
\end{proof}


\section{Upper and lower bounds on \texorpdfstring{$\tau_n$}{tau\_n}}

This section is devoted to the proof of some upper and lower bounds on $\tau_n$.
For all real numbers $x, y \geq 1$, define
\begin{equation}\label{equ:xi-definition}
	\xi_{x, y} := \sum_{x \,<\, i \,<\, y} \frac1{\ell_i} .
\end{equation}
The next lemma provides an estimate of $\tau_n$ in terms of $\xi_{x,y}$.

\begin{lemma}\label{lem:tau-xi-estimate}
	Let $x \geq 1$ be a real number and let $n \geq 1$ be an integer.
	Suppose that $\ell_{\lfloor x \rfloor + 1} > n$.
	Then
	\begin{equation}\label{lem:tau-xi-estimate:equ:1}
		\tau_n
		= \xi_{x, n} - \Theta\big(\tfrac1{2}\xi_{x, n}^2\big) + \Theta(x / n) .
	\end{equation}
\end{lemma}
\begin{proof}
	Let $i > x$ be an integer.
	Then $i \geq \lfloor x \rfloor + 1$ and so $\ell_i \geq \ell_{\lfloor x \rfloor + 1} > n$.
	From Lemma~\ref{lem:l_n-equal-l_mn} it follows that $\ell_n = \ell_{i, n}$. Consequently,
	\begin{equation*}
		L_i(\ell_n) = L_i(\ell_{i, n}) = n < \ell_i ,
	\end{equation*}
	which in turn implies that $\big\{L_i(\ell_n) / \ell_i \big\} = n / \ell_i$.
	Therefore, from Lemma~\ref{lem:prod-vs-sum} and \eqref{equ:xi-definition} it follows that
	\begin{equation}\label{lem:tau-xi-estimate:equ:2}
		\frac1{n} \sum_{x \,<\, i \,<\, n} \frac{\rho_{i+1}}{\rho_n} \left\{\frac{L_i(\ell_n)}{\ell_i}\right\}
		= \sum_{x \,<\, i \,<\, n} \frac1{\ell_i} \prod_{j \,=\, i + 1}^{n - 1} \left(1 - \frac1{\ell_j}\right)
		= \xi_{x, n} - \Theta\big(\tfrac1{2}\xi_{x, n}^2\big) .
	\end{equation}
	Moreover, it is clear that
	\begin{equation}\label{lem:tau-xi-estimate:equ:3}
		0 \leq
		\frac1{n} \sum_{1 \,\leq\, i \,\leq\, x} \frac{\rho_{i+1}}{\rho_n} \left\{\frac{L_i(\ell_n)}{\ell_i}\right\}
		\leq \frac1{n} \sum_{1 \,\leq\, i \,\leq\, x} 1
		\leq \frac{x}{n} .
	\end{equation}
	Combining \eqref{lem:tau-xi-estimate:equ:2} and \eqref{lem:tau-xi-estimate:equ:3} yields \eqref{lem:tau-xi-estimate:equ:1}.
\end{proof}

The following lemma provides an upper bound on $\tau_n$.

\begin{lemma}\label{lem:tau_n-upper-bound}
	Let $c_1 \in [0.99, 1]$, let $n_1 \geq 10771$ be an integer, define
	\begin{align*}
		n_2 &:= \lceil n_1 \log n_1 \rceil , \\
		c_2 &:= -\log(1 / c_1) + \big(1 - (\llog n_1)/\! \log n_1\big)^{-1} ,
	\end{align*}
    and let $c_3$ be given by \eqref{equ:constant-c_3}.
	Suppose that \eqref{equ:ell_n-first-lower-bound} holds for all integers $n \geq n_1$.
	Then \eqref{equ:tau_n-upper-bound} holds for all integers $n \geq n_2$.
\end{lemma}
\begin{proof}
	Let $n \geq n_2$ be an integer and put $x := \omega(n) / c_1$.
	From $c_1 \in [0.99, 1]$, $n \geq n_2$, and Eq.~\eqref{lem:omega-properties:equ:2} of Lemma~\ref{lem:omega-properties} it follows that
	\begin{equation}\label{lem:tau_n-upper-bound:equ:3}
		x
		> \omega(n)
		\geq \omega(n_2)
		\geq \omega(n_1 \log n_1)
		= n_1 .
	\end{equation}
	From \eqref{lem:tau_n-upper-bound:equ:3}, \eqref{equ:ell_n-first-lower-bound}, and Eq.\ \eqref{lem:omega-properties:equ:1} of Lemma~\ref{lem:omega-properties} it follows that
	\begin{equation}\label{lem:tau_n-upper-bound:equ:4}
		\ell_{\lfloor x \rfloor + 1}
		> c_1 x \log x
		= \omega(n) \log\!\left(\frac{\omega(n)}{c_1}\right)
		\geq \omega(n) \log \omega(n)
		= n .
	\end{equation}
	Hence \eqref{lem:tau_n-upper-bound:equ:4} and Lemma~\ref{lem:tau-xi-estimate} imply estimate \eqref{lem:tau-xi-estimate:equ:1}, which in turn gives
	\begin{equation}\label{lem:tau_n-upper-bound:equ:5}
		\tau_n
		\geq \xi_{x, n} - \tfrac1{2}\xi_{x, n}^2 .
	\end{equation}
	Furthermore, from \eqref{lem:tau-xi-estimate:equ:1}, Eq.\ \eqref{lem:omega-properties:equ:3} of Lemma~\ref{lem:omega-properties}, and Lemma~\ref{lem:log-t-to-a-over-t} (since $n_2 > \neper^\neper$) it follows that
	\begin{equation}\label{lem:tau_n-upper-bound:equ:6}
		\tau_n
		\leq \xi_{x, n} + \frac{x}{n}
		\leq \xi_{x, n} + \frac1{c_1 (\log n - \llog n)}
		\leq \xi_{x, n} + \frac1{c_1}\!\left(1 - \frac{\llog n_1}{\log n_1}\right)^{-1}\!\!\frac1{\log n} .
	\end{equation}
	Since $c_1 \in [0.99, 1]$ and
	\begin{equation}\label{lem:tau_n-upper-bound:equ:6-and-a-half}
		n_2 \geq n_1 \log n_1 \geq 10771 \log 10771 > 10^5 ,
	\end{equation}
    Lemma~\ref{lem:omega-sum} implies that
    \begin{equation}\label{lem:tau_n-upper-bound:equ:6-and-three-quarters}
        \sum_{x \,<\, k \,<\, n} \frac1{k \log k}
        = \frac{\llog n - \log(1/c_1)}{\log n} + \Theta\!\left(\!c_3\!\left(\frac{\llog n}{\log n}\right)^{\!2}\right) .
    \end{equation}
    From \eqref{lem:tau_n-upper-bound:equ:3}, \eqref{equ:ell_n-first-lower-bound}, \eqref{equ:xi-definition}, and \eqref{lem:tau_n-upper-bound:equ:6-and-three-quarters} it follows that
	\begin{equation}\label{lem:tau_n-upper-bound:equ:7}
		\xi_{x, n}
		\leq \frac1{c_1} \sum_{x \,<\, k \,<\, n} \frac1{k \log k}
		\leq \frac1{c_1}\!\left(\frac{\llog n - \log(1/c_1)}{\log n} + c_3\!\left(\frac{\llog n}{\log n}\right)^{\!2} \right)
	\end{equation}
	Finally, combining \eqref{lem:tau_n-upper-bound:equ:6} and \eqref{lem:tau_n-upper-bound:equ:7} yields \eqref{equ:tau_n-upper-bound}, as desired.
\end{proof}

The next lemma provides a lower bound on $\tau_n$.

\begin{lemma}\label{lem:tau_n-lower-bound}
	Let $c_1 \in [0.99, 1]$, $c_2, c_4 > 0$, let $n_2 \geq 10771$ be an integer, define
	\begin{align*}
		n_3 &:= \lceil n_2 \log n_2 \rceil , \\
		c_5 &:= \log(1/c_1), \\
		c_6 &:= \frac1{c_1} \!\left(\frac1{2} + \frac{c_2}{\llog n_3} + \frac{c_4}{(\llog n_3)^2}\right)\!\left(1 - \frac{\llog n_3}{\log n_3}\right)^{-1} \\
		&\hspace{2em}+ \frac1{2c_1^2}\!\left(\frac1{\llog n_3} + \frac{2c_3}{\log n_3} + \frac{c_3^2 \llog n_3}{(\log n_3)^2}\right) ,
	\end{align*}
    and let $c_3$ be given by \eqref{equ:constant-c_3} but with $n_3$ in place of $n_2$.
	Suppose that \eqref{equ:ell_n-first-lower-bound} and \eqref{equ:ell_n-upper-bound} hold for all integers $n \geq n_2$.
	Then \eqref{equ:tau_n-lower-bound} holds for all integers $n \geq n_3$.
\end{lemma}
\begin{proof}
	Continue from the end of the proof of Lemma~\ref{lem:tau_n-upper-bound}.
	Make the following change of notation: $n_1, n_2$ of Lemma~\ref{lem:tau_n-upper-bound} are now $n_2, n_3$, respectively.
	(Note that this is consistent with the condition $n_2 \geq 10771$, the formula $n_3 := \lceil n_2 \log n_2 \rceil$, and the change of notation in $c_3$.)

	Let $n \geq n_3$ be an integer.
	Recall that $x := \omega(n) / c_1$ and define
	\begin{equation}\label{lem:tau_n-upper-bound:equ:8}
		y := \frac{\tfrac1{2}(\llog x)^2 + c_2 \llog x + c_4}{c_1 \log x} .
	\end{equation}
	If $k$ is an integer such that $x < k < n$, then from \eqref{lem:tau_n-upper-bound:equ:3}, \eqref{equ:ell_n-upper-bound}, and Lemma~\ref{lem:log-t-to-a-over-t} (since $n_2 > \neper^{\neper^2}$) it follows that
	\begin{align}\label{lem:tau_n-bounds:equ:9}
		\frac1{\ell_k}
		&> \left(1 + \frac{\tfrac1{2}(\llog k)^2 + c_3 \llog k + c_4}{c_1\log k}\right)^{-1} \frac1{k \log k} \\
		&\geq (1 + y)^{-1} \frac1{k \log k}
		\geq (1 - y) \frac1{k \log k} . \nonumber
	\end{align}
	Furthermore, from \eqref{lem:tau_n-upper-bound:equ:8}, the inequalities $\omega(n) \leq x < n$, and Eq.\ \eqref{lem:omega-properties:equ:4} of Lemma~\ref{lem:omega-properties} it follows that
	\begin{align}\label{lem:tau_n-bounds:equ:10}
		y &\leq \frac{\tfrac1{2}(\llog n)^2 + c_2 \llog n + c_4}{c_1 \log \omega(n)}
		\leq \frac{\tfrac1{2}(\llog n)^2 + c_2 \llog n + c_4}{c_1 (\log n - \llog n)} \\
		&\leq \frac{(\llog n)^2}{c_1 \log n}\!
		\left(\frac1{2} + \frac{c_2}{\llog n} + \frac{c_4}{(\llog n)^2}\right)\!\left(1 - \frac{\llog n}{\log n}\right)^{-1} . \nonumber
	\end{align}
	In turn, from \eqref{equ:xi-definition}, \eqref{lem:tau_n-bounds:equ:9}, \eqref{lem:tau_n-upper-bound:equ:6-and-a-half}, and Lemma~\ref{lem:omega-sum} (since $c_1 \in [0.99, 1]$ and $n_3 > 10^5$ by \eqref{lem:tau_n-upper-bound:equ:6-and-a-half}) it follows that
	\begin{equation}\label{lem:tau_n-bounds:equ:11}
		\xi_{x, n}
		\geq (1 - y) \sum_{x \,<\, k \,<\, n} \frac1{k \log k}
		\geq (1 - y)\frac{\llog n - \log(1/c_1)}{\log n}
        \geq \frac{\llog n - \log(1/c_1)}{\log n} - \frac{y \llog n}{\log n} .
	\end{equation}
	Furthermore, from~\eqref{lem:tau_n-upper-bound:equ:7} it follows that
	\begin{equation}\label{lem:tau_n-bounds:equ:12}
		\xi_{x, n}^2
		\leq \frac1{c_1^2}\!\left(\frac{\llog n}{\log n} + c_3\!\left(\frac{\llog n}{\log n}\right)^{\!2} \right)^{\!2}
		= \frac{(\llog n)^3}{c_1^2 (\log n)^2} \left(\frac1{\llog n} + \frac{2c_3}{\log n} + \frac{c_3^2 \llog n}{(\log n)^2}\right) .
	\end{equation}
	Finally, putting together \eqref{lem:tau_n-upper-bound:equ:5}, \eqref{lem:tau_n-bounds:equ:11}, \eqref{lem:tau_n-bounds:equ:12}, and \eqref{lem:tau_n-bounds:equ:10} yields \eqref{equ:tau_n-lower-bound}, as desired.
\end{proof}


\section{Upper and lower bounds on \texorpdfstring{$\ell_n$}{ell\_n}}\label{sec:ell_n-bounds}

The following lemma gives an upper bound on $\ell_n$.

\begin{lemma}\label{lem:ell_n-upper-bound}
	Let $c_1, c_2, c_3 > 0$, let $n_2 > \neper^\neper$ be an integer, and define
	\begin{align}
		r_1 &:= (1 - 1/\ell_{n_2})^{-1} , \nonumber\\
		r_2 &:= \left(1 - \frac1{c_1} \!\left(\frac{\llog n_2 + c_2}{\log n_2} + c_3 \!\left(\frac{\llog n_2}{\log n_2}\right)^{\!2} \right)\right)^{-1} , \nonumber\\
		c_4 &:= c_1(\varrho_{n_2} + \gamma + 1)
		-\tfrac1{2}\big(\!\llog (n_2-1)\big)^2 - c_2\llog (n_2-1)
		+\frac{r_1 r_2}{(n_2 - 1)\log n_2} \label{equ:constant-c_4} \\
		&\hspace{1em}+\int_{\log(n_2-1)}^{+\infty} \Bigg(c_3\!\left(\frac{\log s}{s}\right)^{\!2} + \frac{r_2}{c_1}\!\left(\frac{\log s + c_2}{s} + c_3\!\left(\frac{\log s}{s}\right)^{\!2}\right)^{\!\!2}\Bigg) \mathrm{d}s \nonumber
	\end{align}
	Suppose that \eqref{equ:ell_n-first-lower-bound} and \eqref{equ:tau_n-upper-bound} hold for all integers $n \geq n_2$.
	Then \eqref{equ:ell_n-upper-bound} holds for all integers $n \geq n_2$.
\end{lemma}
\begin{proof}
	If $k \geq n_2$ is an integer, then \eqref{equ:tau_n-upper-bound} and Lemma~\ref{lem:log-t-to-a-over-t} (since $n_2 > \neper^\neper$) imply that
	\begin{equation}\label{lem:ell_n-upper-bound:equ:1}
		\frac1{(1 - 1/\ell_k)(1 - \tau_k)}
		\leq \left(1 + \frac{r_1}{\ell_k}\right) \frac1{1 - \tau_k}
		\leq \frac1{1 - \tau_k} + \frac{r_1 r_2}{\ell_k}
		\leq 1 + \tau_k + r_2 \tau_k^2 + \frac{r_1 r_2}{\ell_k} .
	\end{equation}
	Let $n \geq n_2$ be an integer.
	From Eq.~\eqref{lem:varrho_n-basics:equ:2} of Lemma~\ref{lem:varrho_n-basics} and \eqref{lem:ell_n-upper-bound:equ:1} it follows that
	\begin{equation}\label{lem:ell_n-upper-bound:equ:2}
		\varrho_n
		= \varrho_{n_2} + \sum_{k \,=\, n_2}^{n - 1} \frac1{k}\left(\frac1{(1 - 1/\ell_k)(1 - \tau_k)} - 1\right)
		\leq \varrho_{n_2} + \sum_{k \,=\, n_2}^{n - 1} \frac1{k} \left(\tau_k + r_2 \tau_k^2 + \frac{r_1 r_2}{\ell_k}\right) .
	\end{equation}
	Now the goal is to bound the three parts of the last sum in \eqref{lem:ell_n-upper-bound:equ:2}.
	First, from \eqref{equ:tau_n-upper-bound} and Lemma~\ref{lem:log-t-to-a-over-t} (since $n_2 > \neper^\neper$) it follows that
	\begin{align}\label{lem:ell_n-upper-bound:equ:3}
		\sum_{k \,=\, n_2}^{n - 1} \frac{\tau_k}{k}
		&\leq \sum_{k \,=\, n_2}^{n - 1} \frac1{c_1 k} \!\left(\frac{\llog k + c_2}{\log k} + c_3 \!\left(\frac{\llog k}{\log k}\right)^{\!2} \right) \\
		&< \frac1{c_1}\int_{n_2-1}^n \frac{\llog t + c_2}{t \log t} \,\mathrm{d}t + \frac{c_3}{c_1} \int_{n_2-1}^{+\infty} \left(\frac{\llog t}{\log t}\right)^{\!2} \frac{\mathrm{d}t}{t} \nonumber \\
		&= \frac{\left[\tfrac1{2}(\llog t + c_2)^2\right]_{t \,=\, n_2 - 1}^n}{c_1} + \frac{c_3}{c_1} \int_{\log(n_2-1)}^{+\infty} \left(\frac{\log s}{s}\right)^{\!2} \mathrm{d}s . \nonumber
	\end{align}
	Second, a similar reasoning yields
	\begin{align}\label{lem:ell_n-upper-bound:equ:4}
		\sum_{k \,=\, n_2}^{n - 1} \frac{\tau_k^2}{k}
		&< \frac1{c_1^2}\int_{\log(n_2 - 1)}^{+\infty} \!\left(\frac{\log s + c_2}{s} + c_3\!\left(\frac{\log s}{s}\right)^{\!2}\right)^{\!\!2} \mathrm{d}s.
	\end{align}
	Third and last, from \eqref{equ:ell_n-first-lower-bound} it follows that
	\begin{align}\label{lem:ell_n-upper-bound:equ:5}
		\sum_{k \,=\, n_2}^{n - 1} \frac1{k \ell_k}
		&\leq \sum_{k \,=\, n_2}^{n - 1} \frac1{c_1 k^2 \log k}
		< \frac1{c_1 \!\log n_2}\sum_{k \,=\, n_2}^{+\infty} \frac1{k^2}  \\
		&\leq \frac1{c_1 \!\log n_2} \int_{n_2 - 1}^{+\infty} \frac{\mathrm{d}t}{t^2}
		= \frac1{c_1 (n_2 - 1) \log n_2} . \nonumber
	\end{align}
	Combining Eq.\ \eqref{lem:varrho_n-basics:equ:1} of Lemma~\ref{lem:varrho_n-basics}, \eqref{lem:ell_n-upper-bound:equ:2}, \eqref{lem:ell_n-upper-bound:equ:3}, \eqref{lem:ell_n-upper-bound:equ:4}, and \eqref{lem:ell_n-upper-bound:equ:5} yields that
	\begin{align}\label{lem:ell_n-upper-bound:equ:6}
		\rho_n
		&< \log n + \frac{\tfrac1{2} (\llog n)^2 + c_2 \llog n + c_4}{c_1}
	\end{align}
	Finally, upper bound \eqref{equ:ell_n-upper-bound} follows from \eqref{lem:ell_n-upper-bound:equ:6} and Lemma~\ref{lem:l_mn-formula}.
\end{proof}

\begin{remark}\label{rmk:c_4-integral}
	The integral in \eqref{equ:constant-c_4} has an elementary primitive (which is not written here, since it is quite long) and so its computation poses no difficulties.
\end{remark}

The next lemma provides a lower bound on $\ell_n$.

\begin{lemma}\label{lem:ell_n-second-lower-bound}
	Let $c_1, c_2, c_3, c_5, c_6 > 0$, let $n_3 > \neper^{\neper^2}$ be an integer, and define
	\begin{align}
		r_3 &:= \max\!\Big\{0, -\varrho_{n_3} - \gamma - 1 + \frac{0.542}{n_3} \label{equ:constant-r3} \\
		&+ \tfrac1{2}(\llog n_3)^2 - c_5 \llog(n_3 - 1) + c_6 \int_{\log(n_3 - 1)}^{+\infty} \frac{(\log s)^3}{s^2}\,\mathrm{d}s \Big\} ,  \nonumber\\
		r_4 &:= c_2 + \frac{27\neper^{c_2/3 - 3}}{2} + c_3\!\left(\frac{(\llog n_3)^2}{\log n_3} + \frac1{2}\!\left(\frac{(\llog n_3)^2}{\log n_3}\right)^{\!\!2} \,\right) , \nonumber \\
		c_7 &:= c_5 + \frac1{c_1} , \nonumber \\
		c_8 &:= r_3 + \frac{r_4}{c_1} . \nonumber
	\end{align}
	Suppose that inequalities \eqref{equ:tau_n-upper-bound} and \eqref{equ:tau_n-lower-bound} hold for all integers $n \geq n_3$.
	Then \eqref{equ:ell_n-second-lower-bound} holds for all integers~$n \geq n_3$.
\end{lemma}
\begin{proof}
	Let $n \geq n_3$ be an integer.
	From \eqref{equ:tau_n-lower-bound} and inequality \eqref{lem:varrho_n-basics:equ:3} of Lemma~\ref{lem:varrho_n-basics} it follows that
	\begin{equation}\label{lem:ell_n-second-lower-bound:equ:1}
		\varrho_n
		\geq \varrho_{n_3} + \sum_{k \,=\, n_3}^{n - 1} \frac1{k}\!\left(\frac{\llog k - c_5}{\log k} - \frac{c_6(\llog k)^3}{(\log k)^2} \right) .
	\end{equation}
	On the one hand, Lemma~\ref{lem:log-t-to-a-over-t} (since $n_3 > \neper^\neper$) implies that
	\begin{equation}\label{lem:ell_n-second-lower-bound:equ:2}
		\sum_{k \,=\, n_3}^{n - 1} \frac{\llog k}{k\log k}
		\geq \int_{n_3}^n \frac{\llog t}{t \log t}\,\mathrm{d}t
		= \left[\tfrac1{2} (\llog t)^2\right]_{t \,=\, n_3}^n .
	\end{equation}
	On the other hand, Lemma~\ref{lem:log-t-to-a-over-t} (since $n_3 > \neper^{\neper^{3/2}}$) implies that
	\begin{align}\label{lem:ell_n-second-lower-bound:equ:3}
		\sum_{k \,=\, n_3}^{n - 1} & \frac1{k}\!\left(\frac{c_5}{\log k} + \frac{c_6(\llog k)^3}{(\log k)^2} \right)
		< c_5 \int_{n_3 - 1}^n \frac{\mathrm{d}t}{t \log t} + c_6 \int_{n_3 - 1}^{+\infty} \frac{(\llog t)^3}{t (\log t)^2} \,\mathrm{d}t  \\
		&= c_5 \left[\llog t\right]_{t \,=\, n_3 - 1}^n + c_6 \int_{\log(n_3 - 1)}^{+\infty} \frac{(\log s)^3}{s^2}\,\mathrm{d}s \nonumber
	\end{align}
	Combining \eqref{lem:ell_n-second-lower-bound:equ:1}, \eqref{lem:ell_n-second-lower-bound:equ:2}, \eqref{lem:ell_n-second-lower-bound:equ:3}, and Eq.~\eqref{lem:varrho_n-basics:equ:1} of Lemma~\ref{lem:varrho_n-basics} yields
	\begin{equation}\label{lem:ell_n-second-lower-bound:equ:4}
		\rho_n \geq \log n + \tfrac1{2}(\llog n)^2 - c_5 \llog n - r_3 .
	\end{equation}
	From Lemma~\ref{lem:l_mn-formula}, \eqref{lem:ell_n-second-lower-bound:equ:4}, and \eqref{equ:tau_n-upper-bound} it follows that
	\begin{align}\label{lem:ell_n-second-lower-bound:equ:5}
		\frac{\ell_n}{n}
		&\geq \big(\!\log n + \tfrac1{2}(\llog n)^2 - c_5 \llog n - r_3\big) \left(1 - \frac1{c_1} \!\left(\frac{\llog n + c_2}{\log n} + c_3 \!\left(\frac{\llog n}{\log n}\right)^{\!2} \right)\right) \\
		&> \log n + \tfrac1{2}(\llog n)^2 - c_7 \llog n \nonumber \\
		&\phantom{3em}- \left(r_3 + \frac1{c_1}\!\left(c_2 + \frac{(\llog n)^2 (\llog n + c_2)}{2 \log n} + \frac{c_3(\llog n)^2}{\log n} + \frac{c_3}{2}\!\left(\frac{(\llog n)^2}{\log n}\right)^{\!\!2}\right)\right) . \nonumber
	\end{align}
	If $t \geq 1$ is a real number, then the inequality of arithmetic and geometric means and Lemma~\ref{lem:log-t-to-a-over-t} imply that
	\begin{equation}\label{lem:ell_n-second-lower-bound:equ:6}
		\frac{(\log t)^2 (\log t + c_2)}{t}
		\leq \frac{(\log t + c_2/3)^3}{t}
		\leq \neper^{c_2/3} \max_{s \,\geq\, 1} \frac{(\log s)^3}{s}
		= 27 \neper^{c_2/3 - 3} .
	\end{equation}
	Finally, from \eqref{lem:ell_n-second-lower-bound:equ:5}, \eqref{lem:ell_n-second-lower-bound:equ:6}, and Lemma~\ref{lem:log-t-to-a-over-t} (since $n_3 > \neper^{\neper^2}$) inequality \eqref{equ:ell_n-second-lower-bound} follows.
\end{proof}

\begin{remark}
	Similarly to Remark~\ref{rmk:c_4-integral}, the integral in \eqref{equ:constant-r3} has an elementary primitive (which is not written here, since it is quite long) and so its computation poses no difficulties.
\end{remark}

\section{Numerical computations}\label{sec:numerical-computations}

This section concerns the computation of the explicit constants of Theorems~\ref{thm:lower} and~\ref{thm:lower}.
The corresponding source code is available in a public repository~\cite{Repo} and works in two phases.
First, the highly efficient \textsf{C} code written by Barco~\cite{Barco} and based on the work of Bille, Christiansen, Prezza, and Skjoldjensen~\cite{MR3713292} computes a table of lucky numbers $\ell_n$ up to $n = \ConMaxLuckyIndex$.
Second, a \textsf{SageMath}~\cite{SageMath} script uses the previous table and interval arithmetic to compute the explicit constants with certified precision.
Hereafter, a question mark at the end of a decimal expansion means that the preceding digit may have an error of $\pm 1$.

\subsection{First lower bound}

Put $n_0 := \ConFlbNzer$.
Let $c_1$ and $n_1$ be defined as in Lemma~\ref{lem:ell_n-first-lower-bound}.
Then a computation shows that
\begin{equation*}
	c_1 = \ConRoneCone, \quad
	n_1 = \ConFlbNone ,
\end{equation*}
and Lemma~\ref{lem:ell_n-first-lower-bound} implies that \eqref{equ:ell_n-first-lower-bound} holds for all integers $n \geq n_1$.

\subsection{First round}\label{subsec:first-round}

Redefine $n_1 := \ConRoneNone$ (this is the maximum integer such that later \mbox{$n_3 < \ConMaxLuckyIndex$}).
Of course, inequality \eqref{equ:ell_n-first-lower-bound} still holds for all integers $n \geq n_1$.

Let $n_2$, $c_2$, $c_3$ be defined as in Lemma~\ref{lem:tau_n-upper-bound}.
Then
\begin{equation*}
	n_2 = \ConRoneNtwo , \quad
	c_2 = \ConRoneCtwo , \quad
	c_3 = \ConRoneCthr , \quad
\end{equation*}
and Lemma~\ref{lem:tau_n-upper-bound} implies that \eqref{equ:tau_n-upper-bound} holds for all integers $n \geq n_2$.

Let $c_4$ be defined as in Lemma~\ref{lem:ell_n-upper-bound}.
Then
\begin{equation*}
	c_4 = \ConRoneCfou
\end{equation*}
and Lemma~\ref{lem:ell_n-upper-bound} implies that \eqref{equ:ell_n-upper-bound} holds for all integers $n \geq n_2$.

Let $n_3$, $c_5$, and $c_6$ be defined as in Lemma~\ref{lem:tau_n-lower-bound}.
Then
\begin{equation*}
	n_3 = \ConRoneNthr , \quad
	c_5 = \ConRoneCfiv , \quad
	c_6 = \ConRoneCsix , \quad
\end{equation*}
and Lemma~\ref{lem:tau_n-lower-bound} implies that \eqref{equ:tau_n-lower-bound} holds for all integers $n \geq n_3$.

Finally, let $c_7$ and $c_8$ be defined as in Lemma~\ref{lem:ell_n-second-lower-bound}.
Then
\begin{equation*}
	c_7 = \ConRoneCsev , \quad
	c_8 = \ConRoneCeig , \quad
\end{equation*}
and Lemma~\ref{lem:ell_n-second-lower-bound} implies that \eqref{equ:ell_n-second-lower-bound} holds for all integers $n \geq n_3$.

\subsection{Bootstrapping}

Define
\begin{equation*}
	n_4 := \Big\lceil \!\exp\!\Big(\!\exp\!\Big(c_7 + (c_7^2 + 2c_8)^{1/2}\Big)\Big) \Big\rceil .
\end{equation*}
Note that
\begin{equation}\label{equ:bootstrapping}
    \tfrac1{2}(\llog n)^2 - c_7 \llog n - c_8 \geq 0
\end{equation}
for all integers $n \geq n_4$.
Since $n_3 < n_4 < 10^{100}$, from \eqref{equ:bootstrapping} and \eqref{equ:ell_n-second-lower-bound} it follows that \eqref{thm:lower:equ:1} holds for all integers $n \geq 10^{100}$.

Now the goal is to prove that \eqref{thm:lower:equ:1} is true for all positive integers $n < 10^{100}$.
Employ the notation of Lemma~\ref{lem:ell_n-finite-range}.
A simple analysis shows that taking
\begin{equation}\label{lem:ell_n-first-numerical-lower-bound:equ:2}
	t = \frac{W(\neper^{1-c_0}) + c_0 - 1}{c_0} ,
\end{equation}
maximizes $m_2$.
Applying Lemma~\ref{lem:ell_n-finite-range} with $n_0 = 66$, resp.\ $n_0 = 124000$, and $t$ given by \eqref{lem:ell_n-first-numerical-lower-bound:equ:2} yields that \eqref{thm:lower:equ:1} is true for $n \in [1269, 31807212]$, resp.\ $n \in [28824381, 1.2... \!\cdot\! 10^{100}]$.
Furthermore, a quick computation shows that \eqref{thm:lower:equ:1} is true for all positive integers $n \leq 1269$.
This proves that \eqref{thm:lower:equ:1} holds for all positive integers $n < 10^{100}$.
Therefore, inequality \eqref{thm:lower:equ:1} is true for all integers $n \geq 1$, as desired.

\subsection{Second round}

Since \eqref{thm:lower:equ:1} holds for all integers $n \geq 1$, it is possible to repeat the computations of Subsection~\ref{subsec:first-round} using $c_1^\prime := 1$ and $n_1^\prime := \ConRtwoNone$ instead.
It follows that
\begin{equation*}
	\begin{gathered}
		n_2^\prime = \ConRtwoNtwo , \quad
		c_2^\prime = \ConRtwoCtwo , \quad
		c_3^\prime = \ConRtwoCthr , \quad \\
		c_4^\prime = \ConRtwoCfou \\
		n_3^\prime = \ConRtwoNthr , \quad
		c_5^\prime = \ConRtwoCfiv , \quad
		c_6^\prime = \ConRtwoCsix , \quad \\
		c_7^\prime = \ConRtwoCsev , \quad
		c_8^\prime = \ConRtwoCeig , \quad
	\end{gathered}
\end{equation*}
and that \eqref{thm:lower:equ:2} holds for all integers $n \geq n_3^\prime$.
This proves \eqref{thm:lower:equ:2} and completes the proof of Theorem~\ref{thm:lower}.

\subsection{Third half-round}

Let $c_1^{\prime\prime} := 1$ and $n_1^{\prime\prime} := \ConRthrNone$.
Repeating the computations of Subsection~\ref{subsec:first-round} that led to the upper bound of $\ell_n$ gives
\begin{equation*}
    \begin{gathered}
        n_2^{\prime\prime} = \ConRthrNtwo , \quad
        c_2^{\prime\prime} = \ConRthrCtwo , \quad
        c_3^{\prime\prime} = \ConRthrCthr , \quad \\
        c_4^{\prime\prime} = \ConRthrCfou ,
    \end{gathered}
\end{equation*}
and that \eqref{thm:upper:equ:2} holds for all integers $n > \ConMaxLuckyIndex$.
Finally, a direct computation shows that \eqref{thm:upper:equ:1} holds for all integers $n \in [4, \ConMaxLuckyIndex]$.
This proves Theorem~\ref{thm:upper}.

\providecommand{\bysame}{\leavevmode\hbox to3em{\hrulefill}\thinspace}
\providecommand{\MR}{\relax\ifhmode\unskip\space\fi MR }
\providecommand{\MRhref}[2]{%
	\href{http://www.ams.org/mathscinet-getitem?mr=#1}{#2}
}
\providecommand{\href}[2]{#2}

\end{document}